\documentclass[graybox]{svmult}

\usepackage{mathptmx}       
\usepackage{helvet}         
\usepackage{courier}        
\usepackage{type1cm}        
\usepackage{makeidx}         
\usepackage{graphicx}        
\usepackage{caption}
\usepackage{multicol}        
\usepackage[bottom]{footmisc}
\usepackage{amsmath, amsfonts, amssymb}
\usepackage{color}
\usepackage{array,multirow,makecell}

\let\pa\partial  
\let\na\nabla  
\let\eps\varepsilon

\newcommand{\T}{{\mathcal T}}
\newcommand{\Pt}{{\mathcal P}}

\newcommand{\F}{{\mathcal F}}
\newcommand{\dist}{{\mathrm d}}
\newcommand{\m}{{\mathrm m}}

\newcommand{\N}{{\mathbb N}}  
\newcommand{\R}{{\mathbb R}} 
 
\newcommand{\diver}{\operatorname{div}}

\makeindex             

\begin{document}

\title*{A finite-volume scheme for a cross-diffusion model arising from
interacting many-particle population systems}
\titlerunning{A finite-volume scheme for a cross-diffusion system}
\author{Ansgar J\"ungel and Antoine Zurek}

\institute{
\at
Institute for Analysis and Scientific Computing, Vienna University of
	Technology, Wiedner Hauptstra\ss e 8--10, 1040 Wien, Austria \\
\email{juengel@tuwien.ac.at, antoine.zurek@tuwien.ac.at}}

\maketitle

\abstract{A finite-volume scheme for a cross-diffusion model arising from
the mean-field limit of an interacting particle system for multiple population
species is studied. 
The existence of discrete solutions and a discrete entropy production inequality
is proved. The proof is based on a weighted quadratic entropy that is not the 
sum of the entropies of the population species.
\newline\indent
\keywords{Finite volume scheme, cross-diffusion system, 
entropy method.
\\[5pt]
}
}


\section{Introduction}


\subsection{Presentation of the model}

We consider the following cross-diffusion system:
\begin{align}\label{1.eq}
  \pa_t u_i + \diver\big(-\delta\na u_i - u_i\na p_i(u)\big) = 0, \quad
	p_i(u) = \sum_{j=1}^n a_{ij}u_j \quad\mbox{in }\Omega,\ t>0, 
\end{align}
where $i=1,\ldots,n$ with $n\ge 2$, $\Omega\subset\R^2$ is an open bounded polygonal 
domain, and $\delta > 0$, $a_{ij} > 0$. We impose 
the initial and no-flux boundary conditions
\begin{equation}\label{1.bic}
  u_i(0)=u_i^0\ge 0\quad\mbox{in }\Omega,\quad 
	\na u_i\cdot\nu=0\quad\mbox{on }\pa\Omega,\ t>0,\ i=1,\ldots,n,
\end{equation}
where $\nu$ is the exterior unit normal vector on $\pa\Omega$. We write 
$u:=(u_1,\ldots,u_n)$ and $u^0:=(u^0_1,\ldots,u^0_n)$. 
Equations \eqref{1.eq} are derived from a weakly interacting stochastic many-particle
system in the mean-field limit \cite{CDJ19}. It can be seen as a simplification
of the Shigesada-Kawasaki-Teramoto (SKT) population model \cite{SKT79}, where the
diffusion is reduced to $\delta\na u_i$. The two-species system was analyzed
first in \cite{BGHP85}, but up to now, no analytical or numerical results are
available for the $n$-species system.
The diffusion matrix associated to \eqref{1.eq} is neither symmetric nor positive
definite but we show below that system \eqref{1.eq} possesses an entropy
structure \cite{Ju15} yielding gradient estimates that are the basis for the
numerical analysis.

We assume that $(a_{ij})\in\R^{n\times n}$
is positively stable (i.e., all eigenvalues of $A=(a_{ij})$ have positive real parts)
and that the detailed-balance condition holds, i.e., 
there exist numbers $\pi_1,\ldots,\pi_n>0$ such that 
\begin{equation}\label{1.db}
  \pi_i a_{ij} = \pi_j a_{ji} \quad\mbox{for all }i,j=1,\ldots,n.
\end{equation}
Note that for the two-species model this condition is always satisfied, 
just set $\pi_1=a_{21}$ and $\pi_2=a_{12}$. Since $A_1=\operatorname{diag}(\pi_i^{-1})$
is symmetric, positive definite and $A_2=(\pi_i a_{ij})$ is symmetric, by
\cite[Prop.~6.1]{Ser10}, the number of positive eigenvalues of $A=A_1A_2$ equals
that for $A_2$. Thus, $A_2$ has only positive eigenvalues, which together with
the symmetry means that $A_2$ is symmetric, positive definite. 

Our (numerical) analysis is based on the observation that system \eqref{1.eq}
possesses an entropy structure with a weighted quadratic entropy that has not been
observed before in cross-diffusion systems:
$$
  H[u]=\int_\Omega h(u)dx, \quad\mbox{where }
	h(u) := \frac{1}{2\delta}\sum_{i,j=1}^n \pi_i a_{ij} u_i u_j.
$$
Interestingly, this entropy is not of the form $\sum_{i=1}^n h_i(u_i)$, but it
mixes the species. A formal computation shows that 
$$
  \frac{dH}{dt} + \sum_{i,j=1}^n \pi_i a_{ij}\int_\Omega\na u_i\cdot\na u_j dx
	+ \frac{1}{\delta}\sum_{i=1}^n\pi_i\int_\Omega u_i|\na p_i(u)|^2 dx = 0.
$$
With $\lambda>0$ being the smallest eigenvalue of $(\pi_ia_{ij})$, we conclude
the following entropy production inequality:
$$
  \frac{dH}{dt} + \lambda\sum_{i=1}^n\int_\Omega|\na u_i|^2 dx
	+ \frac{1}{\delta}\sum_{i=1}^n\pi_i\int_\Omega u_i|\na p_i(u)|^2 dx \le 0.
$$
Our aim is to prove this inequality for the finite-volume solutions.


\subsection{The numerical scheme}

A mesh of $\Omega$ is given by a set $\T$ of open
polygonal control volumes, a set $\mathcal{E}$ of edges, and a set $\Pt$ of points
$(x_K)_{K\in\T}$. We assume that the mesh is admissible in the sense of 
Definition 9.1 in \cite{EGH00}.
We distinguish in $\mathcal{E}$ the interior edges 
$\sigma=K|L$ and the exterior edges such that
$\mathcal{E}=\mathcal{E}_{\mathrm{int}} \cup \mathcal{E}_{\mathrm{ext}}$. 
For a given control volume $K\in\T$, we denote by $\mathcal{E}_K$ the
set of its edges. This set splits into $\mathcal{E}_K=\mathcal{E}_{{\rm int},K}
\cup \mathcal{E}_{{\rm ext},K}$. 
For any $\sigma\in\mathcal{E}$, there exists at least one cell
$K\in\T$ such that $\sigma\in\mathcal{E}_K$ and we denote this cell by $K_\sigma$. 
When $\sigma$ is an interior edge, $\sigma=K|L$, $K_\sigma$ can be either $K$ or $L$. 
For all $\sigma\in\mathcal{E}$, we define $\dist_\sigma = \dist(x_K,x_L)$ if 
$\sigma = K|L \in \mathcal{E}_{{\rm int}}$ and $\dist_\sigma = \dist(x_K,\sigma)$ 
if $\sigma \in \mathcal{E}_{{\rm ext},K}$. Then the transmissibility coefficient 
is defined by $\tau_\sigma = \m(\sigma)/\dist_\sigma$ for all $\sigma \in \mathcal{E}$. 
We assume that the mesh satisfies the following regularity constraint: 
\begin{equation}\label{2.regmesh}
  \exists \xi > 0,\, \forall K\in\T,\, \forall\sigma\in\mathcal{E}_K: \
	\dist(x_K,\sigma) \ge \xi\dist_\sigma.
\end{equation}
The size of the mesh is denoted by $\Delta x=\max_{K\in\T}\operatorname{diam}(K)$.
Let $N_T\in\N$ be the number of time steps, $\Delta t=T/N_T$ be the
time step size, and $t_k=k\Delta t$ for $k=0,\ldots,N_T$.

Let $\mathcal{H}_\T$ be the linear space of functions $\Omega \to \R$ which are constant on each $K \in \T$. For $v \in \mathcal{H}_\T$, we introduce
\begin{equation*}
  D_{K,\sigma}v = v_{K,\sigma}-v_K, \quad D_\sigma v = |D_{K,\sigma} v| 
	\quad \mbox{for all }K \in \T, \,\sigma \in \mathcal{E}_K,
\end{equation*}
where $v_{K,\sigma}$ is either $v_L$ ($\sigma=K|L$) or $v_K$ ($\sigma \in 
\mathcal{E}_{{\rm ext},K}$). Finally, we define the (squared) discrete $H^1$ norm
$$
	\|v\|^2_{1,2,\T} = \sum_{\sigma \in \mathcal{E}} \tau_\sigma (D_\sigma v)^2 
	+ \sum_{K \in \T} \m(K)v_K^2.
$$
For all $K \in \T$ and $i=1,\ldots,n$, $u^0_{i,K}$ denotes the mean value of $u^0_i$ 
over $K$. The finite-volume scheme for \eqref{1.eq} reads as
\begin{align}\label{sch1}
  & \frac{\m(K)}{\Delta t}(u_{i,K}^{k}-u_{i,K}^{k-1})
	+ \sum_{\sigma\in\mathcal{E}_K}\F_{i,K,\sigma}^{k} = 0, \quad i=1,\ldots,n, \\
	& \F_{i,K,\sigma}^{k} = -\tau_\sigma \big(\delta D_{K,\sigma} u^k_i 
	+ u^k_{i,\sigma} D_{K,\sigma}  p_i (u^k)\big) \quad\mbox{for all }K \in \T, \, 
	\sigma \in \mathcal{E}_K, \label{sch2}
\end{align}
with $u^k =(u^k_1, \ldots, u^k_{n})$ and 
$u_{i,\sigma}^{k} := \min\{u^k_{i,K},u^k_{i,K,\sigma}\}$.
As in \cite{ABB11}, this definition of $u_{i,\sigma}^{k}$ allows us to prove 
the nonnegativity of $u_{i,K}^k$.


\subsection{Main result}

The main result of this work is the existence of nonnegative solutions to scheme 
\eqref{sch1}-\eqref{sch2}, which preserve the entropy production inequality.

\begin{theorem}[Existence of discrete solutions]\label{thm.ex}
Assume that $u^0 \in L^2(\Omega)^n$ with $u^0_i \geq 0$, $\delta>0$, $a_{ij} > 0$, 
$(a_{ij})$ is positively stable, and \eqref{1.db} holds. Then there exists a solution 
$(u_K^k)_{K\in \T, \,k=0,\ldots,N_T}$ with $u_K^k=(u_{1,K}^k,\ldots,u_{n,K}^k)$ 
to scheme \eqref{sch1}-\eqref{sch2} satisfying $u_{i,K}^k\ge 0$ for all	$K\in\T$, 
$i=1,\ldots,n$, and $k=0,\ldots,N_T$. Moreover, the following discrete 
entropy production inequality holds:
\begin{align}
  \sum_{K\in \T} \m(K) h(u^{k}_K) 
	&+ \Delta t \lambda \sum_{i=1}^n 
	\sum_{\sigma \in \mathcal{E}} \tau_\sigma (D_{\sigma} u^k_i)^2 \nonumber \\ 
	&{}+ \frac{\Delta t}{\delta}\sum_{i=1}^n \sum_{\sigma \in \mathcal{E}} \tau_\sigma 
	\pi_i u^k_{i,\sigma} (D_{\sigma} p_i(u^k))^2  
	\le \sum_{K \in \T} \m(K)  h(u^{k-1}_K), \label{4.edi1} 
\end{align}
where $\lambda$ denotes the smallest eigenvalue of $(\pi_i a_{ij})$.
\end{theorem}

We expect that the detailed-balance condition \eqref{1.db} can be replaced by a
weak cross-diffusion condition as in \cite{CDJ18}. The positive stability of
$(a_{ij})$ implies the parabolicity of \eqref{1.eq} in the sense of Petrovskii.
Indeed, $(\pi_i a_{ij})$ and $\operatorname{diag}(u_i/\pi_i)$
are symmetric, positive definite matrices for $u\in(0,\infty)^n$. 
Thus, its product $(u_ia_{ij})$ has only positive eigenvalues \cite[Theorem 7]{Bos87}
which proves the claim. The assumption that the diffusion coefficient $\delta$
is the same for all species is a simplification needed to conclude that
$h(u)$ is coercive, $h(u)\ge(\lambda/2)|u|^2$ for $u\in\R^n$. It can
be removed by exploiting the Shannon entropy to show first that $u_i$ is nonnegative, 
but this requires more technical effort which will be detailed in a future work.


\section{Proof of Theorem~\ref{thm.ex}}

We proceed by induction. 
For $k=0$, we have $u_i^0 \geq 0$ by assumption. Assume that there exists a solution
$u^{k-1}$ for some $k\in\{2,\ldots,N_T\}$ such that $u_i^{k-1} \geq 0$ in $\Omega$,
$i=1,\ldots,n$. The construction of a solution $u^k$ is split in several steps.

{\em Step 1: Definition of a linearized problem.} Let $R > 0$, we set
$$
  Z_R := \big\{w=(w_{1}, \ldots, w_{n}) \in (\mathcal{H}_\T)^n \ : \ \|w_{i}\|_{1,2,\T}< R 
	\quad\mbox{for }i=1,\ldots,n\big\},
$$
and let $\eps>0$ be given. We define the mapping $F_\eps:Z_R\to\R^{\theta n}$
by $F_\eps(w)=w^\eps$, with $\theta=\#\T$, where 
$w^\eps = (w_{1}^\eps,\ldots,w_{n}^\eps)$ is the solution to the linear problem
\begin{equation}\label{3.lin}
  \eps\left(\sum_{\sigma\in\mathcal{E}_K}\tau_\sigma D_{K,\sigma}(w_i^\eps) 
	+ \m(K) w^\eps_{i,K} \right)
  = -\bigg(\frac{\m(K)}{\Delta t}(u_{i,K}-u_{i,K}^{k-1})
	+ \sum_{\sigma\in\mathcal{E}_K}\F^+_{i,K,\sigma} \bigg),
\end{equation}
for $K\in\T$, $i=1,\ldots,n$, and $\F^+_{i,K,\sigma}$ is defined in \eqref{sch2}
with $u_{i,\sigma}$ replaced by $\bar u_{i,\sigma}=\min\{u_{i,K}^+,u_{i,K,\sigma}^+\}$,
where $z^+=\max\{0,z\}$. Here, $u_{i,K}$ is a function of $w_{1,K},\ldots,w_{n,K}$, 
defined by the entropy variables
\begin{equation}\label{3.w}
  w_{i,K} = \frac{\pi_i}{\delta} p_i(u_K)
	= \sum_{j=1}^n \frac{\pi_ia_{ij}}{\delta}u_j
	\quad\mbox{for all }K \in \T, \ i=1,\ldots,n.
\end{equation}
This is a linear system with the invertible coefficient matrix $(\pi_ia_{ij}/\delta)$,
and so, the function $u_K=u(w_K)$ is well-defined.
The existence of a unique solution $w_{i}^\eps$ to the linear scheme 
\eqref{3.lin}-\eqref{3.w} is now a consequence of \cite[Lemma 3.2]{EGH00}.

{\em Step 2: Continuity of $F_\eps$.} We fix $i\in\{1,\ldots,n\}$. 
Multiplying \eqref{3.lin} by $w_{i,K}^\eps$ and summing over $K\in\T$, we obtain, 
after discrete integration by parts,
$$
  \eps\|w^\eps_{i}\|^2_{1,2,\T}
	= -\sum_{K \in \T}\frac{\m(K)}{\Delta t}(u_{i,K}-u_{i,k}^{k-1}) w^{\eps}_{i,K} 
	+ \sum_{\substack{\sigma\in\mathcal{E}_{\mathrm{int}} \\ \sigma=K|L}}
	\F^+_{i,K,\sigma} D_{K,\sigma} w_{i}^\eps  
	=: J_1 + J_2.
$$
By the Cauchy-Schwarz inequality and the definition of $\F^+_{i,K,\sigma}$, 
we find that
\begin{align*}
  |J_1| &\le \frac{1}{\Delta t}\bigg(\sum_{K\in\T}\m(K)(u_{i,K}-u_{i,K}^{k-1})^2
	\bigg)^{1/2}\bigg(\sum_{K\in\T}\m(K)(w_{i,K}^\eps)^2\bigg)^{1/2} \\
	|J_2| &\le \bigg(\sum_{\sigma\in\mathcal{E}}\tau_\sigma
	\big(\delta D_\sigma u_i + \bar u_{i,\sigma}D_\sigma p_i(u) \big)^2\bigg)^{1/2}
	\bigg(\sum_{\sigma\in\mathcal{E}}\tau_\sigma(D_\sigma w_i^\eps)^2\bigg)^{1/2}.
\end{align*}
Hence, since $u_i$ is a linear combination of $(w_1,\ldots,w_n) \in Z_R$, there exists a constant $C(R)>0$ which is independent of $w^\eps$ such that
$|J_1|+|J_2|\le C(R)\|w_{i}^\eps\|_{1,2,\T}$. 
Inserting these estimations, it follows that $
  \eps\|w^{\eps}_{i}\|_{1,2,\T} \le C(R)$.

We turn to the proof of the continuity of $F_\eps$. Let $(w^m)_{m\in\N}\subset Z_R$ 
be such that $w^m \to w$ as $m\to\infty$. The previous estimate shows that $w^{\eps,m} := F_\eps(w^m)$ is bounded uniformly in $m\in\N$. Thus, there exists a 
subsequence of $(w^{\eps,m})$, which is not relabeled, such that 
$w^{\eps,m} \to w^\eps$ as $m\to\infty$. Passing to the limit $m\to\infty$ in 
scheme \eqref{3.lin}-\eqref{3.w} and taking into account the continuity of the 
nonlinear functions, we see that $w_{i}^\eps$ is a solution to \eqref{3.lin}-\eqref{3.w}
for $i=1,\ldots,n$ and $w^\eps = F_\eps(w)$. Because of the uniqueness of the limit 
function, the whole sequence converges, which proves the continuity.

{\em Step 3: Existence of a fixed point.} We claim that the map $F_\eps$ admits a 
fixed point. We use a topological degree argument \cite{Dei85}, i.e., we prove that 
$\operatorname{deg}(I-F_\eps,Z_R,0) = 1$, where
$\operatorname{deg}$ is the Brouwer topological degree. Since $\mathrm{deg}$ is invariant by homotopy, it is sufficient to prove that any 
solution $(w^\eps,\rho)\in \overline{Z}_R\times[0,1]$ to the fixed-point equation 
$w^\eps = \rho F_\eps(w^\eps)$ satisfies $(w^\eps,\rho)\not\in\pa Z_R\times[0,1]$ 
for sufficiently large values of $R>0$. Let $(w^\eps,\rho)$ be a fixed point and 
$\rho\neq 0$, the case $\rho=0$ being clear. Then $w_{i}^\eps/\rho$ solves
\begin{equation}\label{3.lin2}
  \eps\left(\sum_{\sigma\in \mathcal{E}_K}\tau_\sigma D_{K,\sigma}(w_i^\eps) 
	+\m(K) w^\eps_{i,K}\right)
  = -\rho\bigg(\frac{\m(K)}{\Delta t}(u^\eps_{i,K}-u_{i,K}^{k-1})
	+ \sum_{\sigma\in\mathcal{E}_K}\F^{+,\eps}_{i,K,\sigma}  \bigg),
\end{equation}
for all $K \in \T$, $i=1,\ldots,n$, and $\F_{i,K,\sigma}^{+,\eps}$ is defined as in 
\eqref{sch2} with $u$ replaced by $u^\eps$. The following discrete entropy 
production inequality is the key argument.

\begin{lemma}[Discrete entropy production inequality]\label{lem.edi1}
Let the assumptions of Theorem \ref{thm.ex} hold. Then, for any $\rho\in(0,1]$ and
$\eps>0$,
\begin{align}
   \rho \sum_{K \in \T} & \m(K) h(u^\eps_K) 
	+ \eps\Delta t\sum_{i=1}^n \|w^\eps_{i}\|^2_{1,2,\T}
	+ \rho \Delta t \lambda \sum_{i=1}^n \sum_{\sigma \in \mathcal{E}}
	\tau_\sigma(D_\sigma u_i^\eps)^2 \nonumber \\
	&{}+ \rho \frac{\Delta t}{\delta} \sum_{i=1}^n 
	\sum_{\sigma \in \mathcal{E}}\tau_\sigma 
	\pi_i \bar u^{\eps}_{i,\sigma} (D_\sigma p_i(u^\eps))^2  
	\le \rho \sum_{K \in \T} \m(K) h(u^{k-1}_K),\label{edi1.dis}
\end{align}
with $\lambda>0$ being the smallest eigenvalue of $(\pi_i a_{ij})$ and 
obvious notations for $\bar u_{i,\sigma}^{\eps}$.
\end{lemma}

\begin{proof}
We multiply \eqref{3.lin2} by $\Delta t w_{i,K}^\eps$ and sum over $i$
and $K\in\T$. This gives, after discrete integration by parts,
$\eps\Delta t\sum_{i=1}^n \|w_{i}^\eps\|^2_{1,2,\T}	+ J_3 + J_4 + J_5=0$, where 
\begin{align*}
	J_3 &= \rho\sum_{i=1}^n\sum_{K\in\T}\m(K)(u_{i,K}^\eps-u_{i,K}^{k-1})w_{i,K}^\eps, \\
	J_4 &= -\rho \Delta t \sum_{i=1}^n\sum_{\substack{\sigma\in\mathcal{E}_{\mathrm{int}} 
	\\ \sigma=K|L}} \tau_\sigma \delta D_{K,\sigma} u_i^\eps w_{i,K}^\eps, \\
	J_5 &= \rho \Delta t \sum_{i=1}^n\sum_{\substack{\sigma\in\mathcal{E}_{\mathrm{int}} 
	\\ \sigma=K|L}} \tau_\sigma\bar u^{\eps}_{i,\sigma} D_{K,\sigma} p_i(u^\eps) 
	D_{K,\sigma} w_{i,K}^\eps.
\end{align*}
For $J_3$, we use the convexity of $h$ for its estimation; 
for $J_4$, we take into account the symmetry of $\tau_\sigma$ with 
respect to $\sigma=K|L$, definition~\eqref{3.w} of $w^\eps_{i}$ and the positive 
definiteness of $(\pi_i a_{ij})$; and for $J_5$, we employ
definition~\eqref{3.w} of $w^\eps_{i}$:
\begin{align*}
  J_3 &\geq \rho \sum_{K \in \T} \m(K) \big(h(u_K^\eps)-h(u_K^{k-1})\big), \\
	J_4 &= \rho \Delta t \sum_{i,j=1}^n \sum_{\substack{\sigma\in\mathcal{E}_{\mathrm{int}}
	\\ \sigma=K|L}} \tau_\sigma \pi_i a_{ij} D_{K,\sigma} u^\eps_{i} 
	D_{K,\sigma} u^\eps_{j}
  \geq \rho \Delta t \lambda \sum_{i=1}^n \sum_{\sigma \in \mathcal{E}} 
	\tau_\sigma (D_{\sigma} u^\eps_{i})^2, \\
	 J_5 &= \rho\frac{\Delta t}{\delta}\sum_{i=1}^n \sum_{\sigma \in \mathcal{E}} 
	\tau_\sigma \pi_i\bar u^\eps_{i,\sigma}(D_\sigma p_i(u^\eps))^2.
\end{align*}
Putting all the estimations together completes the proof.\qed
\end{proof}

We proceed with the topological degree argument. Lemma \ref{lem.edi1} implies that
$$
  \eps\Delta t\sum_{i=1}^n\|w_{i}^\eps\|^2_{1,2,\T}
	\le \rho \sum_{K\in \T} \m(K) h(u^{k-1}_K) \le \sum_{K\in \T} \m(K) 
	h(u^{k-1}_K).
$$
Then, if we define $R:=(\eps\Delta t)^{-1/2}(\sum_{K\in\T}\m(K)h(u^{k-1}_K))^{1/2}+1$,
we conclude that $w^\eps \not\in\pa Z_R$ and $\operatorname{deg}(I-F_\eps,Z_R,0)=1$. 
Thus, $F_\eps$ admits a fixed point

{\em Step 4: Limit $\eps\to 0$.} 
Recall that $h(u_K)\ge (\lambda/2)|u_K|^2$ (note that $u_{i,K}\in\R$ at this point). 
Thus, by Lemma~\ref{lem.edi1}, there exists a constant $C>0$ 
depending only on the mesh but not on $\eps$ such that for all $K \in \T$ and 
$i=1,\ldots,n$,
$$
  |u_{i,K}^\eps| \le C(\lambda)\left(\sum_{K \in \T} \m(K) h(u^{k-1}_K) \right)^{1/2}.
$$
Thus, up to a subsequence, for $i=1,\ldots,n$ and for all $K \in \T$, we infer the 
existence of $u_{i,K}\in\R$ such that $u^\eps_{i,K} \to u_{i,K}$ as $\eps\to 0$. 
We deduce from \eqref{edi1.dis} that there exists a subsequence (not relabeled) such 
that $\eps w_{i,K}^\eps\to 0$ for any $K\in\T$ and $i=1,\ldots,n$. 
Hence, the limit $\eps\to 0$ in~\eqref{3.lin2} yields the existence of a solution to
\eqref{3.lin} with $\eps=0$.

Let $i\in\{1,\ldots,n\}$ and $K\in\T$ such that $u_{i,K}=\min_{L\in\T}u_{i,L}$.
We multiply \eqref{3.lin} with $\eps=0$ by $\Delta t u_{i,K}^-$ with 
$z^-=\min\{0,z\}$ and use the induction hypothesis:
\begin{align*}
  \m(K)(u_{i,K}^-)^2 &- \Delta t\sum_{\sigma\in\mathcal{E}_K}\tau_\sigma
	(\delta + a_{ii}\bar u_{i,\sigma})D_{K,\sigma}(u_i) u_{i,K}^-  \\
	&{}- \Delta t\sum_{j\neq i}\sum_{\sigma\in\mathcal{E}_K}\tau_\sigma a_{ij}
	\bar u_{i,\sigma}D_{K,\sigma}(u_j)u_{i,K}^- = 0.
\end{align*}
The second term is nonpositive since $\bar u_{i,\sigma}\ge 0$ and
$D_{K,\sigma}(u_i)\ge 0$, by the choice of $K$. The last term vanishes since
$\bar u_{i,\sigma}u_{i,K}^- = u_{i,K}^+u_{i,K}^-=0$, by the definition of
$\bar u_{i,\sigma}$.
This shows that $u_{i,L}\ge u_{i,K}\ge 0$ for all $L\in\T$ and $i=1,\ldots,n$.
Passing to the limit $\eps \to 0$ 
in~\eqref{edi1.dis} yields inequality~\eqref{4.edi1}, 
which completes the proof of Theorem \ref{thm.ex}.


\section{Convergence analysis and perspectives}

In this section, we sketch the proof of the convergence of the scheme and possible 
extensions of the method presented in this paper.

\begin{itemize}
\item[$\bullet$] Let us give the main features of the proof of convergence. First, 
thanks to the a priori estimates given by~\eqref{4.edi1} and 
assumption~\eqref{2.regmesh}, we prove the existence of a constant $C>0$ independent 
of $\Delta x$ and $\Delta t$ such that for all $i=1,\ldots,n$ and 
$\phi \in C^\infty_0(\Omega \times (0,T))$,
\begin{equation}\label{transl.time}
  \sum_{k=1}^{N_T} \sum_{K \in \T} \m(K) (u^k_{i,K}-u^{k-1}_{i,K})\phi(x_K,t_k) 
	\leq C \|\nabla \phi\|_{L^\infty(\Omega \times (0,T))}.
\end{equation}
Next, we consider a sequence of admissible meshes $(\T_\eta,\Delta t_\eta)_{\eta >0}$ 
of $\Omega \times (0,T)$, indexed by the size $\eta = \{\Delta x, \Delta t \}$, 
satisfying~\eqref{2.regmesh} uniformly in $\eta$. For any $\eta > 0$, we denote by 
$u_\eta=(u_{1,\eta},\ldots,u_{n,\eta})$ the piecewise constant (in time and space) 
finite-volume solution constructed in Theorem~\ref{thm.ex}. We deduce, 
thanks to~\cite[Theorem 3.9]{ACM17} and~\eqref{transl.time}, that there exist 
nonnegative functions $u_1, \ldots, u_n$ such that, up to a subsequence, 
$$
  u_{i,\eta} \to u_i \quad \mbox{a.e. in }\Omega \times (0,T) \mbox{ as }\eta \to 0,
	\quad i=1,\ldots,n.
$$
Moreover, we conclude from~\eqref{4.edi1} that $u_{i,\eta} \in L^\infty(0,T;L^2(\Omega))
\subset L^2(\Omega\times (0,T))$ uniformly in $\eta$ for $i=1,\ldots,n$. Hence, 
$(u_{i,\eta})$ is equi-integrable in $L^2(\Omega \times (0,T))$. Thus, applying 
the Vitali convergence theorem, we deduce that, up to a subsequence,
$u_{i,\eta} \to u_i$ strongly in $L^2(\Omega \times (0,T))$ as $\eta \to 0$, 
$i=1,\ldots,n$. 
The discrete entropy production inequality yields a uniform bound of the discrete 
gradient $\na^\eta$ of $u_{i,\eta}$ in $L^2(\Omega \times (0,T))$; see~\cite{CLP03} 
for a definition of $\na^\eta$. It follows from \cite[Lemma 4.4]{CLP03} that, up to 
a subsequence,
$$
  \na^\eta u_{i,\eta} \rightharpoonup \na u_i \quad \mbox{weakly in }
	L^2(\Omega \times (0,T))\mbox{ as }\eta \to 0, \ i=1,\ldots,n.
$$ 
Finally, following the method developed in~\cite{CLP03}, we prove that the limit 
function $u=(u_1,\ldots,u_n)$ is a weak solution to~\eqref{1.eq}-\eqref{1.bic}.

\item[$\bullet$] 
We already mentioned that system \eqref{1.eq} can be interpreted as a simplification
of the SKT model. In a future work, we will analyze a structure-preserving
finite-volume approximation of the full SKT model. 
Such a discretization was analyzed in \cite{ABB11}, but only for positive
definite diffusion matrices associated to \eqref{1.eq}. We will extend the
analysis of \cite{ABB11} without this assumption.
\end{itemize}

\begin{acknowledgement}
The authors acknowledge partial support from the Austrian Science Fund (FWF), 
grants P30000, P33010, W1245, and F65.
\end{acknowledgement}

\end{document}